\documentclass[a4paper,12pt]{article}
\usepackage[utf8]{inputenc}
\usepackage{amsthm}
\usepackage{amsfonts}
\usepackage{amssymb}
\usepackage{hyperref}
\usepackage{graphicx}
\usepackage{rotating}
\usepackage{wrapfig}
\usepackage[all]{xy}
\usepackage{pgf,tikz}
\usetikzlibrary{arrows}
\usetikzlibrary[patterns]

\newcommand{\N}{\mathbb{N}}

\newcommand{\R}{\mathbb{R}}

\newcommand{\B}{\mathfrak{B}}

\newtheorem{prop}{Proposition}[section]
\newtheorem{cor}[prop]{Corollary}

\newtheorem{teo}[prop]{Theorem}

\newtheorem{lema}[prop]{Lemma}

\theoremstyle{definition}

\newtheorem{defi}[prop]{Definition}
\newtheorem*{obs}{Remark}
\newtheorem*{acknowledgements}{Acknowledgements}

\title{Characterization of the $n$-dimensional Sierpi\'nski carpet as an inverse limit of closed balls}
\author{Lucas H. R. de Souza}

\begin{document}

\DeclareGraphicsExtensions{.pdf,.jpg,.mps,.png,}

\maketitle

\def\eod{\hfill$\square$}

\begin{abstract}In this paper we generalize, for any dimension, a theorem of Tshishiku and Walsh that characterizes the Sierpi\'nski carpet as a limit of a set of maps from the disc to the sphere. 
\end{abstract}

\textbf{Keywords} \ Sierpi\'nski carpet $\cdot$ inverse limit

\

\textbf{Mathematics Subject Classification (2010)} \ Primary 54F65 $\cdot$ 54F15 Secondary 54F17 $\cdot$ 54E45

\


\let\thefootnote\relax\footnote{BELO HORIZONTE, MG, BRAZIL.}

\let\thefootnote\relax\footnote{\textit{E-mail address:} \textbf{henrique.lucas20003@gmail.com} \ \ ORCID: \textbf{0000-0001-5661-8705}}

\tableofcontents

\section*{Introduction}

In \cite{Wh}, Whyburn did two different topological characterizations of the $1$-dimensional Sierpi\'nski carpet. After that, Cannon \cite{Ca} generalized one of these characterizations for any dimension different than $3$. It says that if we take an $n$-dimensional sphere and remove a countable set of tame $n$-balls such that their union is dense and their diameter tends to $0$, then the space that we get does not depend on the choice of the removed balls. Such space is the $n-1$-dimensional Sierpi\'nski carpet.

Recently, Tshishiku and Walsh did another characterization of the $1$-dimensional Sierpi\'nski carpet as an inverse limit of closed discs (Lemma 3.1 of \cite{TW} and its correction in \cite{TW2}) for the purpose of using it to have a topological characterization of some boundaries of groups. Roughly speaking, it says that if we take a sphere, choose a countable dense subset of it and for each of these points we remove it and replace it nicely with a circle, then we get a $1$-dimensional Sierpi\'nski carpet. With a similar purpose, this paper is devoted to generalizing Tshishiku and Walsh result for any dimension. It consists of the following:

\

$\!\!\!\!\!\!\!\!\!\!\!\!$ \textbf{Theorem \ref{blowupissierpinski}} Let $P$ be a countable dense subset of the sphere $S^{n}$. Let, for every $p \in P$, $\pi_{p}: D^{n} \rightarrow S^{n}$ be the map that collapses the boundary of the $n$-ball to the point $p$. Then the limit space $\lim\limits_{\longleftarrow} \{D^{n},\pi_{p}\}_{p \in P}$ is homeomorphic to the $n-1$-dimensional Sierpi\'nski carpet.

\

Constructions of the $n-1$-dimensional Sierpi\'nski carpet as inverse limits of $n$-balls are well known in the literature, as seen in \cite{AO} (proof of Theorem 1 (1)), \cite{BNW} (proof of Theorem 3) and \cite{BO} (proof of Theorem 1.2). However, as far as I know, all of these constructions are special cases of the theorem above, since the inverse limits described there depend on extra structures, like well behaved homeomorphisms or group actions. 

In \cite{So5} we will use this characterization to construct some Bowditch boundaries of relatively hyperbolic groups that are homeomorphic to the $n$-dimensional Sierpi\'nski carpet. For instance, if $(G,\mathcal{P})$ is a relatively hyperbolic pair with its Bowditch boundary homeomorphic to $S^{3}$, $\mathcal{Q}$ is a proper subset of $\mathcal{P}$ such that every group in $\mathcal{P}-\mathcal{Q}$ is virtually torsion-free and hyperbolic, then the Bowditch boundary of $(G,\mathcal{Q})$ is homeomorphic to the $2$-dimensional Sierpi\'nski carpet (Theorem 2.8 of \cite{So5}).

\

\

\

\begin{acknowledgements}This paper contains part of my PhD thesis. It was written under the advisorship of Victor Gerasimov, to whom I am grateful. I would also like to thank to Jan Boro\'nski (who showed me the references \cite{AO}, \cite{BNW} and \cite{BO}) and Kazuhiro Kawamura (who showed me that Freedman's theorem would solve the problem on the restriction of Cannon's theorem in dimension $4$) for the nice conversations that I had with each of them about this paper.
\end{acknowledgements}

\section{Preliminaries}

Here we list some well known topological results that are used in the next section.

Some notation that we use in this paper:

\begin{enumerate}
    \item If $X$ is is a topological space and $Y \subseteq X$, then the closure of $Y$ is denoted by $\bar{Y}$ or $Cl(Y)$ and the interior of $Y$ is denoted by $int(Y)$.
    \item If $X$ is a metric space, $Y \subseteq X$ and $\epsilon > 0$, then the $\epsilon$-neighborhood of $Y$ is denoted by $\B(Y,\epsilon)$.
\end{enumerate}

\subsection{Homogeneity}

\begin{defi}A topological space $X$ is strongly locally homogeneous if $\forall x \in X$, $\forall U$ neighbourhood of $x$, there exists an open set $V$ such that $x \in V \subseteq U$ and $\forall y \in V$, there exists a homeomorphism $f: X \rightarrow X$ such that $f(x) = y$ and $f|_{X-U} = id_{X-U}$. \end{defi}

\begin{defi}A topological space $X$ is countable dense homogeneous if for every $A,B$ countable dense subsets of $X$, there exists a homeomorphism $f: X \rightarrow X$ such that $f(A) = B$.
\end{defi}

\begin{prop}\label{cotabledensehomogeneous}Let $X$ be a metrizable locally compact strongly locally homogeneous space and $(\bar{X},d)$ a metric compactification of $X$. If $A$ and $B$ are two countable dense subsets of $X$, there exists a homeomorphism $f: \bar{X} \rightarrow \bar{X}$ such that $f(A) = B$ and $f|_{\bar{X}-X} = id_{\bar{X}-X}$.
\end{prop}

\begin{obs}This proposition slightly generalizes Bennett's Theorem (Theorem 3 of \cite{Ben}) and the proof is essentially the same. We write it here for the sake of completeness.
\end{obs}

\begin{proof} Let's construct a family of homeomorphisms $\{f_{n}\}_{n \in \N}$, using a back and forth argument, that converges to the homeomorphism that we need.

Let $A = \{a_{i}\}_{i\in \N}$ and $B = \{b_{i}\}_{i\in \N}$. Let $i_{1}\in \N$ such that $Cl_{\bar{X}}(\B(a_{1},\frac{1}{2^{i_{1}}})) \subseteq X$. Since $X$ is strongly locally homogeneous, there exists $V_{1} \subseteq \B(a_{1},\frac{1}{2^{i_{1}}})$ an open set and $b \in V_{1}\cap B$ such that there is a homeomorphism $h_{1}: X \rightarrow X$ with $h_{1}(a_{1}) = b$ and $h_{1}|_{X-\B(a_{1},\frac{1}{2^{i_{1}}})} = id_{X-\B(a_{1},\frac{1}{2^{i_{1}}})}$. Since $Cl_{\bar{X}}(\B(a_{1},\frac{1}{2^{i_{1}}}))$ do not intersect $\bar{X}-X$, the map $h_{1}$ extends to a homeomorphism $f_{1}: \bar{X} \rightarrow \bar{X}$ such that $f_{1}|_{\bar{X}-X} = id_{\bar{X}-X}$.

Suppose we have a homeomorphism $f_{2n}: \bar{X} \rightarrow \bar{X}$ such that $f_{2n}|_{\bar{X}-X} = id_{\bar{X}-X}$, $f_{2n}(\{a_{1},...,a_{n}\}) \subseteq B$ and $f^{-1}_{2n}(\{b_{1},...,b_{n}\}) \subseteq A$. If $f_{2n}(a_{n+1}) \in B$, take $f_{2n+1} = f_{2n}$. Suppose that $f_{2n}(a_{n+1}) \notin B$. There is an index $i_{2n+1} \geqslant 2n+1$ such that $Cl_{\bar{X}}\B(f_{2n}(a_{n+1}),\frac{1}{2^{i_{2n+1}}})$ does not intersect the set $\{f_{2n}(a_{1}),...,f_{2n}(a_{n}),b_{1},...,b_{n}\} \cup (\bar{X}-X)$. Then there exists an open set $V_{2n+1}$ such that $f_{2n}(a_{n+1}) \in V_{2n+1} \subseteq \B(f_{2n}(a_{n+1}),\frac{1}{2^{i_{2n+1}}})$ and a homeomorphism $h_{2n+1}: \bar{X} \rightarrow \bar{X}$ such that $h_{2n+1}(f_{2n}(a_{2n+1})) \in V \cap B$ and $h_{2n+1}|_{\bar{X}-\B(f_{2n}(a_{n+1}),\frac{1}{2^{i_{2n+1}}})} = id_{\bar{X}-\B(f_{2n}(a_{n+1}),\frac{1}{2^{i_{2n+1}}})}$. Take $f_{2n+1}= h_{2n+1}\circ f_{2n}$. We have that  $f_{2n+1}|_{\bar{X}-X} = id_{\bar{X}-X}$, $f_{2n+1}|_{\{a_{1},...,a_{n}\}} = f_{2n}|_{\{a_{1},...,a_{n}\}}$, $\forall i \in \{1,...,n+1\}$, $f_{2n+1}(a_{i}) \in B$, $f^{-1}_{2n+1}|_{\{b_{1},...,b_{n}\}} = f^{-1}_{2n}|_{\{b_{1},...,b_{n}\}}$ and $\forall i \in \{1,...,n\}$, $f^{-1}_{2n+1}(b_{i}) \in A$.

Suppose we have a homeomorphism $f_{2n+1}: \bar{X} \rightarrow \bar{X}$ such that $f_{2n+1}|_{\bar{X}-X} = id_{\bar{X}-X}$, $f_{2n+1}(\{a_{1},...,a_{n+1}\}) \subseteq B$ and $f^{-1}_{2n+1}(\{b_{1},...,b_{n}\}) \subseteq A$. If $f^{-1}_{2n+1}(b_{n+1}) \in A$, take $f_{2n+2} = f_{2n+1}$.  Suppose that $f^{-1}_{2n+1}(b_{n+1}) \notin A$. There exists $i_{2n+2} \in \N$ such that $2n+2 \leqslant i_{2n+2}$ and the set  $Cl_{\bar{X}}(\B(f^{-1}_{2n+1}(b_{n+1}),\frac{1}{2^{i_{2n+2}}}))$ do not intersect $\{f^{-1}_{2n+1}(b_{1}),...,f^{-1}_{2n+1}(b_{n}),a_{1},...,a_{n+1}\} \cup (\bar{X}-X)$. Then there exists an open set $V_{2n+2}$ such that $f^{-1}_{2n+1}(b_{n+1}) \in V_{2n+2} \subseteq \B(f^{-1}_{2n+1}(b_{n+1}),\frac{1}{2^{i_{2n+2}}})$ and a homeomorphism $h_{2n+2}: \bar{X} \rightarrow \bar{X}$ such that $h^{-1}_{2n+2}(f^{-1}_{2n+1}(b_{n+1})) \in V \cap A$ and $h_{2n+2}|_{\bar{X}-\B(f^{-1}_{2n+1}(b_{n+1}),\frac{1}{2^{i_{2n+2}}})} = id_{\bar{X}-\B(f^{-1}_{2n+1}(b_{n+1}),\frac{1}{2^{i_{2n+2}}})}$. Take $f_{2n+2}= f_{2n+1}\circ h_{2n+2}$. We have that  $f_{2n+2}|_{\bar{X}-X} = id_{\bar{X}-X}$, $f_{2n+2}|_{\{a_{1},...,a_{n+1}\}} = f_{2n+1}|_{\{a_{1},...,a_{n+1}\}}$, $\forall i \in \{1,...,n+1\}$, $f_{2n+2}(a_{i}) \in B$, $f^{-1}_{2n+2}|_{\{b_{1},...,b_{n}\}} = f^{-1}_{2n+1}|_{\{b_{1},...,b_{n}\}}$ and $\forall i \in \{1,...,n+1\}$, $f^{-1}_{2n+2}(b_{i}) \in A$.

We have that, if $n \leqslant m$, $\rho(f_{m},f_{n}) \leqslant \sum_{j = n}^{m}\frac{1}{2^{i_{j}-1}}$, where $\rho$ is the uniform distance. Since $\{\sum_{j = n}^{\infty}\frac{1}{2^{i_{j}}}\}_{n \in \N}$ converges to $0$, $\{f_{n}\}_{n \in \N}$ is a Cauchy sequence, which implies that it converges uniformly to a continuous map $f: \bar{X} \rightarrow \bar{X}$. Analogously, $\{f^{-1}_{n}\}_{n \in \N}$ converges uniformly to a continuous map $g: \bar{X} \rightarrow \bar{X}$. Since $\forall n \in \N$, $f_{n}|_{\bar{X}-X} = id_{\bar{X}-X}$, we have that $f|_{\bar{X}-X} = id_{\bar{X}-X}$. We have also that $\forall i \in \N$, there exists $n \in \N$ such that $\forall n' > n$, $f_{n'}(a_{i}) = f_{n}(a_{i}) \in B$ and $f^{-1}_{n'}(b_{i}) = f^{-1}_{n}(b_{i}) \in A$, which implies that $f(a_{i}) = f_{n}(a_{i})$ and $g(b_{i}) = f^{-1}_{n}(b_{i})$. So $f(A) = B$, $g(B) = A$, $g \circ f|_{A} =id_{A}$ and $f \circ g|_{B} = id_{B}$. Since $A$ and $B$ are dense on $\bar{X}$, then the maps $f$ and $g$ are inverses. Thus $f$ is the homeomorphism that we need.
\end{proof}

\begin{obs}We have that manifolds with no boundary are strongly homogeneous, which implies that a manifold $M$ with boundary have the property that for every $A,B$ countable dense subsets of $M$ that do not intersect the boundary,  there exists a homeomorphism $f: M \rightarrow M$ such that $f(A) = B$. In particular, Bennett's Theorem implies that manifolds without boundary are countable dense homogeneous.
\end{obs}

\begin{cor}\label{avoidemptyinterior}Let $X$ be a metrizable locally compact strongly locally homogeneous space and $(\bar{X},d)$ a metric compactification of $X$. If $A$ and $B$ are two subsets of $X$ such that $A$ is countable dense and $B$ has empty interior, then there exists a homeomorphism $f: \bar{X} \rightarrow \bar{X}$ such that $f(B) \cap A = \emptyset$ and $f|_{\bar{X}-X} = id_{\bar{X}-X}$.
\end{cor}

\begin{proof}We have that $X-B$ has a countable basis, which implies that it is separable. Let $C$ be a countable dense subset of $X-B$. Since $X-B$ is dense on $X$, then $C$ is dense on $X$. By the last proposition, there is a homeomorphism $f: \bar{X} \rightarrow \bar{X}$ such that $f(C) = A$ and $f|_{\bar{X}-X} = id_{\bar{X}-X}$. This homeomorphism satisfies the property  $f(B) \cap A = \emptyset$.
\end{proof}

\subsection{Approximation of maps}

\begin{prop}\label{fort}Let $M$ and $N$ be two compact metric spaces and $\{F_{i}\}_{i\in \N}$ be a family of maps $F_{i}: M \rightarrow Closed(N)$ satisfying:

\begin{enumerate}
    \item $\forall x \in M$, $\forall i \in \N$, $\forall U$ neighbourhood of $F_{i}(x)$, there exists a neighbourhood $V$ of $x$ such that $\forall x' \in V$, there exists $j > i$ such that $F_{j}(x') \subseteq U$.
    \item $\forall i \in \N$, $\forall x \in M$, $F_{i}(x) \supseteq F_{i+1}(x)$.
    \item $\forall x \in M$, $\lim\limits_{i \in \N} diam(F_{i}(x)) = 0$.
\end{enumerate}

Then the map $f: M \rightarrow N$ defined by $\{f(x)\} = \bigcap_{i \in \N} F_{i}(x)$ is continuous.
\end{prop}

\begin{obs}This proposition slightly generalizes Fort's Theorem about upper semi-continuous maps.
\end{obs}

\begin{proof}By the properties \textbf{2} and \textbf{3} it is clear that f is well defined.

Let $x \in M$ and $U$ a neighbourhood of $f(x)$. Since $\{F_{i}(x)\}_{i \in \N}$ is a nested family of compact spaces with intersection $\{f(x)\}$, there is $i_{0} \in \N$ such that $F_{i_{0}}(x) \subseteq U$ (Proposition 1.7 of \cite{Na}). By the first condition, there exists a neighbourhood $V$ of $x$ such that $\forall x' \in V$, there exists $i_{x'} > i_{0}$ such that $F_{i_{x'}}(x') \subseteq U$. But $f(x') \in F_{i_{x'}}(x')$, which implies that $f(x') \in U$. Thus, $f$ is continuous at $x$, which implies that $f$ is continuous.
\end{proof}

\subsection{Topological quasiconvexity}

\begin{defi}Let $X$ be a compact metric space and $\sim$ an equivalence relation on $X$. We say that $\sim$ is topologically quasiconvex if $\forall q \in X, \ [q]$ is closed and $\forall \epsilon > 0$, $\#\{[x]\subseteq X: diam \ [x] > \epsilon\} < \aleph_{0}$, with $[x]$ the equivalence class of $x$. The partition of $X$ by such equivalence relation is called a null family.

Let $X$, $Y$ be compact metric spaces. A quotient map $f: X \rightarrow Y$ is topologically quasiconvex if the relation $\sim = \Delta^{2}X \cup \bigcup_{y \in Y} f^{-1}(y)^{2}$ is topologically quasiconvex.
\end{defi}

\begin{obs}This definition does not depend on the choice of the metric compatible with the topology of $X$. 
\end{obs}

\begin{prop} (Propositions 1 and 3 of Section 2 of \cite{Dav} and Proposition 1.15 and Corollary 1.16 of \cite{So3}) \label{quasiconvexity} Let $f: X \rightarrow Y$ be a continuous map, where $X$ and $Y$ are Hausdorff compact spaces. For $A \subseteq Y$, let $\sim_{A} = \Delta X \cup \bigcup_{q \notin A} f^{-1}(q)^{2}$ and $X_{A} = X/\!\sim_{A}$. For $p \in Y$, let $\sim_{p} = \sim_{\{p\}}$, $X_{p} = X_{\{p\}}$ and $\pi_{p}: X_{p} \rightarrow Z$ the quotient map. The following are equivalent:

\begin{enumerate}
    \item $f$ is topologically quasiconvex.
    \item $\forall A \subseteq Y$, $X_{A}$ is Hausdorff.
    \item $\forall p \in Y$, $X_{p}$ is Hausdorff.
    \item The induced map $X \rightarrow \lim\limits_{\longleftarrow} \{X_{p},\pi_{p}\}_{p \in Y}$ is a homeomorphism. 
\end{enumerate}

\end{prop}

\subsection{Spheres}

Let $D$ be an open or closed subset of $S^{n}$. We say that $D$ is a tame ball if $Cl(D)$ and $S^{n} - int(D)$ are homeomorphic to closed balls. If $n = 2$ then every interior of a closed disc is tame by Schoenflies Theorem. 

However this is false for $n > 2$ (a counterexample is the Alexander horned sphere), so we need to be careful with it. We say that a subspace $Y$ of $S^{n}$ that is homeomorphic to $S^{n-1}$ is tame in $S^{n}$ if it bounds two closed balls. The Generalized Schoenflies Theorem, says that if the embedding of $Y$ on $S^{n}$ is well behaved, then $Y$ is tame:

\begin{prop}(Generalized Schoenflies Theorem - Brown \cite{Br}) A subset $Y$ of $S^{n}$ that is homeomorphic to $S^{n-1}$ is tame if and only if the inclusion map $Y \rightarrow S^{n}$ extends to an embedding $Y \times [-1,1] \rightarrow S^{n}$ (where we identify $Y$ with $Y \times \{0\}$). 
\end{prop}

If a space $X$ is homeomorphic to $S^{n}$ minus a finite number of tame open balls, we say that a subspace $Y$ of $X$ that is homeomorphic to $S^{n-1}$ is tame in $X$ if it is tame in $X'$, where $X'$ is the space homeomorphic to $S^{n}$ constructed from $X$ by adjoint the balls that are missing. 

We have also a theorem that characterizes a sphere minus two tame discs:

\begin{prop}(Annulus Theorem - Kirby \cite{Ki} for dimension $\neq 4$ and Quinn \cite{Qu} for dimension $4$) The space $S^{n} - (A\cup B)$, where $A$ and $B$ are disjoint tame copies of $S^{n-1}$, has three connected components, which two of then  are balls and the other one has the closure homeomorphic to $S^{n-1} \times[0,1]$.
\end{prop}

The following theorem is well known. It is proved in \cite{Fi} (Theorem 15) for dimension $2$ and $3$. For the other dimensions, it comes as a consequence of the Stable Homeomorphism Theorem (\cite{Ki}, \cite{Qu}) and the second corollary of Theorem 2 of \cite{Ki}.

\begin{prop}(Isotopy Theorem for Spheres) Choose an orientation of $S^{n}$ and let $f,g: S^{n} \rightarrow S^{n}$ be homeomorphisms such that both preserve or reverse orientation. Then there exists an isotopy between $f$ and $g$. 
\end{prop}

\begin{prop}(Decomposition Theorem - Meyer, Theorem 2 of \cite{Me}) Let $C_{1},..., C_{n},...$ be a (possibly finite) family of tame open balls in $S^{n}$ which their closures are pairwise disjoint and such that the equivalence relation $\Delta S^{n} \cup \bigcup_{i} Cl_{S^{n}}C_{i}^{2}$ is topologically quasiconvex.  If $U$ is an open set of $S^{n}$ that contains $\bigcup_{i} Cl_{S^{n}}C_{i}$, then there exists a continuous map $f:S^{n} \rightarrow S^{n}$ that is surjective, it is the identity outside $U$, each $Cl_{S^{n}}C_{i}$ is saturated (i.e. it is the preimage of some point) and $\forall x \in S^{n} - f(\bigcup_{i} Cl_{S^{n}}C_{i})$, $f^{-1}(x)$ is a single point.
\end{prop}

\begin{obs}The version that we stated is in \cite{Ca}. This result is due to Moore (Theorem 25 of \cite{Mo}) when $n=2$.
\end{obs}

As a consequence of all these theorems, it is possible to characterize spheres with a finite set of tame balls removed. The following are easy and probably well known:

\begin{lema}\label{ballseparation} Let $A_{1},A_{2},...,A_{m}$ be disjoint closed tame balls of $S^{n}$, with $n > 1$, and $F$ a closed set of $S^{n}$ such that $S^{n}-F$ is connected and $\forall i \in \{1,...,m\}$, $A_{i} \cap F = \emptyset$. Then there is a closed tame ball $C$ that contains $F$ and $A_{i}\cap C = \emptyset$, $\forall i \in \{1,...,m\}$.
\end{lema}

\begin{proof}By \textbf{Decomposition Theorem}, there is a continuous surjective map $\pi: S^{n} \rightarrow S^{n}$ that collapses only the tame balls $A_{1},...,A_{m}$. Let $X$ be $S^{n} - \{p\}$, where $p \notin \pi(A_{1}\cup....\cup A_{n} \cup F)$. We have that $X$ is homeomorphic to $\R^{n}$ and we choose a metric $d$ on $X$ induced by the euclidean metric on $\R^{n}$. Let $k= diam \ \pi(F)$. Since $X$ is homeomorphic to $\R^{n}$ and $X- \pi(F)$ is connected,  for every pair of distinct $m$-tuples that do not intersect $\pi(F)$, there is a homeomorphism $f: X \rightarrow X$ that send one $m$-tuple to the other and it is the identity map on $\pi(F)$. So we choose the first $m$-tuple as $(\pi(A_{1}),...,\pi_{1}(A_{m}))$ and the second one any $m$-tuple $(x_{1},...,x_{m})$ such that $\forall i \in \{1,...,m\}$, $d(x_{i}, \pi(F)) > k$. If $y \in \pi(F)$, then the closed ball $Cl(\B(y,k))$ is a tame ball that contains $\pi(F)$ and do not intersect $\{x_{1},...,x_{m}\}$. Since $\pi|_{S^{n} - (A_{1}\cup ...\cup A_{m})}: S^{n} - (A_{1} \cup ... \cup A_{m}) \rightarrow S^{n} - \pi(A_{1}\cup ...\cup A_{m})$ is a homeomorphism, then $D = \pi^{-1}(f^{-1}(Cl(\B(y,k))))$ is the tame ball that we are looking for.
\end{proof}

\begin{lema}\label{cylinders}Consider the cylinder $S^{n-1}\times [0,1]$ and a pair of homeomorphisms $f_{i}: S^{n-1}\times\{i\} \rightarrow S^{n-1}\times\{i\}$, for $i \in \{0,1\}$, such that both preserve or reverse orientation (for some orientation of $S^{n-1}\times [0,1]$). Then, there exists a homeomorphism $f: S^{n-1}\times [0,1] \rightarrow S^{n-1}\times [0,1]$ such that $f|_{S^{n-1}\times\{i\}} = f_{i}$.
\end{lema}

\begin{proof}Let $\phi_{i}: S^{n-1}\times\{i\} \rightarrow S^{n-1}$ be the projection maps, for $i \in \{0,1\}$. So the maps $\phi_{1}\circ f_{1}\circ \phi^{-1}_{1} $ and $\phi_{2} \circ f_{2}\circ \phi^{-1}_{2}$ both preserve or reverse orientation. Then, by the Isotopy Classification Theorem, there is an isotopy between the maps $\phi_{1} \circ f_{1}\circ \phi^{-1}_{1}$ and $\phi_{2}\circ f_{2}\circ \phi^{-1}_{2}$: $\eta: S^{n-1} \times [0,1] \rightarrow S^{n-1}$. Then the map  $f: S^{n-1} \times [0,1] \rightarrow S^{n-1} \times [0,1]$ defined as $f(x,t) = (\eta(x,t),t)$ is a homeomorphism. If $x \in S^{n-1}$, then  $f((x,i)) = (\phi_{i} \circ f_{i}\circ \phi^{-1}_{i}(x),i) = \\ (\phi_{i} \circ f_{i}((x,i)),i) = \phi^{-1}_{i} \circ \phi_{i} \circ f_{i}((x,i)) = f_{i}((x,i))$. Thus, $f$ is the homeomorphism that we want.
\end{proof}

\begin{lema}\label{extensionsphere}Let $A_{1},A_{2}$ be disjoint closed tame balls in $S^{n}$ and $f_{i}: A_{i} \rightarrow A_{i}$ be homeomorphisms that both preserve or reverse orientation (for some orientation of $S^{n}$). Then there exists a homeomorphism $f: S^{n} \rightarrow S^{n}$ such that $f|_{A_{i}} = f_{i}$.
\end{lema}

\begin{proof}Immediate from the \textbf{Annulus Theorem} and the last lemma.
\end{proof}

\begin{lema}Let $A,B,C$ be closed tame balls in $S^{n}$ such that $A,B \subseteq C$ and $g: A \rightarrow B$ be a orientation preserving map (for some orientation of $S^{n}$). Then there exists a homeomorphism $f: C \rightarrow C$ such that $f|_{\partial C} = id_{\partial C}$ and $f|_{A} = g$.
\end{lema}

\begin{proof}Let $Y_{A}$ be the connected manifold whose boundary is $\partial A \cup \partial C$ and $Y_{B}$ the connected manifold whose boundary is $\partial B \cup \partial C$. By the Annulus Theorem, the spaces $Y_{A}$ and $Y_{B}$ are both homeomorphic to  $S^{n-1} \times [0,1]$. Since $g$ is an orientation preserving map, there is a homeomorphism $h: Y_{A} \rightarrow Y_{B}$ such such that $h|_{\partial A} = g|_{\partial A}$ and  $h|_{\partial C} = id_{\partial C}$. Thus the map $f: C \rightarrow C$ defined by $f(x) = g(x)$ if $x \in A$ and $f(x) = h(x)$ if $x \in Y_{A}$ is the homeomorphism that we want.
\end{proof}

\begin{prop}\label{characterizationofsphereswithoutballs} Consider the spaces $X = S^{n} - (A_{1}\cup A_{2} \cup ... \cup A_{n})$ and $Y = S^{n} - (B_{1}\cup B_{2} \cup ... \cup B_{n})$, where $\{A_{1},A_{2},...,A_{n}\}$ and $\{B_{1},B_{2},...,B_{n}\}$ are families of disjoint closed tame balls. Let, for every $i \in \{1,...,n\}$,  a homeomorphism $f_{i}: A_{i} \rightarrow B_{i}$ that preserves the orientation of the spheres (induced by a  fixed orientation of $S^{n}$). Then there exists a homeomorphism $f: S^{n} \rightarrow S^{n}$ such that $\forall i \in \{1,...,n\}$, $f|_{A_{i}} = f_{i}$.
\end{prop}

\begin{proof}By \textbf{Lemma \ref{ballseparation}} there exists $C_{1}$ a closed tame ball that contains $A_{1}$ and $B_{1}$ and does not intersect the tame balls $A_{2},...,A_{n},B_{2},...,B_{n}$, We also construct, recursively, tame balls $C_{2},...,C_{n}$ such that $C_{i}$ contains $A_{i}$ and $B_{i}$ and does not intersect the tame balls $C_{1},...,C_{i-1}, A_{i+1},...,A_{n}, B_{i+1},...,B_{n}$. By the last lemma, there exists $f'_{i}: C_{i} \rightarrow C_{i}$ a homeomorphism such that $f'_{i}|_{\partial C_{i}} = id_{\partial C_{i}}$ and $f'_{i}|_{A_{i}} = f_{i}$. Then the map $f: S^{n} \rightarrow S^{n}$ such that $f(x) = f'_{i}(x)$ if $x \in C_{i}$ and $f(x) = x$ if $x \in X- \bigcup_{i = 1}^{n} C_{i}$ is a homeomorphism that restricts to a homeomorphism between $X$ and $Y$.
\end{proof}

\begin{obs}In particular, it shows that the homeomorphism class of the space $S^{n}$ minus $n$ open tame balls does not depend on the choice of these balls removed.
\end{obs}

\subsection{Cellular complexes}

\begin{prop}\label{aproximaçãoportriangulacoes}Let $\{\tilde{T}_{i}\}_{i \in \N}$, $\{\tilde{Q}_{i}\}_{i \in \N}$ be two families of cellular structures of a compact $n$-manifold $M$ (possibly with boundary) such that $\tilde{T}_{i+1}$ is a subdivision of $\tilde{T}_{i}$, $\tilde{Q}_{i+1}$ is a subdivision of $\tilde{Q}_{i}$, every $n$-cell of $\tilde{T}_{i}$ and  $\tilde{Q}_{i}$ has diameter less than $\frac{1}{2^{i}}$. Let $T_{i}$, $Q_{i}$ be the $n-1$-skeletons of $\tilde{T}_{i}$, $\tilde{Q}_{i}$, respectively. Let $\{f_{i}\}_{i \in \N}$, where $f_{i}: (M, \tilde{T}_{i}) \rightarrow (M, \tilde{Q}_{i})$ is a cellular isomorphism which satisfies that $\forall i \in \N$, $\forall j < i$, $f_{i}|_{T_{j}} = f_{j}|_{T_{j}}$. Then there is a homeomorphism $f: M \rightarrow M$ such that $\forall i \in \N$, $f|_{T_{i}} = f_{i}|_{T_{i}}$.
\end{prop}

\begin{proof}Let $F_{i}: M \rightarrow Closed(M)$ defined by $F_{i}(x) = \{f_{i}(x)\}$ if $x \in T_{i}$ and if $x \notin T_{i}$, $F_{i}(x)$ is the $n$-cell of $\tilde{Q}_{i}$ that contains $f_{i}(x)$ (it is unique since $f_{i}$ is a cellular isomorphism, which implies that $f_{i}(x)$ is not on the boundary of a $n$-cell). By the hypothesis on $\{\tilde{Q}_{i}\}_{i \in \N}$, we have that $\forall x \in M$, $\lim\limits_{i \in \N}diam(F_{i}(x)) = 0$. Since $\forall i \in \N$, $f_{i+1}$ extends $f_{i}|_{T_{i}}$ and $\tilde{T}_{i+1}$ is a subdivision of $\tilde{T_{i}}$, it follows that $\forall i \in \N$, $\forall x \in M$, $F_{i+1}(x) \subseteq F_{i}(x)$. Let $x \in M$, $i \in \N$ and $U$ be a neighbourhood of $F_{i}(x)$ (we can suppose that $U$ is open). Let $i'> i$ such that every $n$-cell in $\tilde{Q}_{i'}$ that intersects $F_{i}(x)$ is contained in $U$. Let $V = \bigcup \{D \in \tilde{T}_{i'}:  f_{i'}(D) \subseteq U\}$. Let $D$ be a $n$-cell of $\tilde{T}_{i'}$ such that $x \in D$. Then $f_{i'}(D) \subseteq U$ (by the choice of $i'$), which implies that $D \subseteq V$. Since $D$ is arbitrary, we get that $V$ is a neighbourhood of $x$. Let $x' \in V$. By the construction of $V$ we have that there exists a $n$-cell $D'$ in $\tilde{T}_{i'}$ such that $x' \in D'$ and $f_{i'}(D') \subseteq U$. In any case we have that $F_{i'}(x') \subseteq f_{i'}(D')$, which implies that $F_{i'}(x') \subseteq U$. Then, by \textbf{Proposition \ref{fort}}, there is a continuous map $f: M \rightarrow M$ defined by $\{f(x)\} = \bigcap_{i \in \N} F_{i}(x)$. Since $\forall i \in \N$, $\forall x \in T_{i}$, $F_{i}(x) = \{f_{i}(x)\}$, we have that $\forall i \in \N$, $f$ extends $f_{i}|_{T_{i}}$.

Let $T = \bigcup_{i \in \N} T_{i}$ and $Q = \bigcup_{i \in \N} Q_{i}$. Analogously to the last paragraph, the sequence $\{f_{i}^{-1}\}_{i \in \N}$ converges uniformly to a homeomorphism $g: M \rightarrow M$ such that $\forall i \in \N$, $g|_{Q_{i}} = f^{-1}_{i}|_{Q_{i}}$.  It is clear that $g|_{Q}\circ f|_{T} = id_{T}$ and $f|_{T} \circ g|_{Q} = id_{Q}$. Since $T$ and $Q$ are dense in $M$, we have that $g\circ f = id_{M}$ and $f \circ g = id_{M}$. Thus $f$ is a homeomorphism.
\end{proof}

\section{Sierpi\'nski carpet}
\label{sierpinskicarpetsection}

Consider the Menger space $M^{n}_{n-1}$. It is the Sierpi$\acute{n}$ski carpet of dimension $n-1$. The space $M^{n}_{n-1}$ is constructed by taking $S^{n}$ and removing a family of disjoint open tame balls $\{C_{i}\}_{i \in \N}$ satisfying:

\begin{enumerate}
    \item $\bigcup_{i \in \N} \bar{C_{i}}$ is dense in $S^{n}$.
    \item $\forall \epsilon > 0$, the set $\{i \in \N: diam \  \bar{C_{i}} > \epsilon\}$ is finite.
\end{enumerate}

From now on we are always regarding the Sierpi$\acute{n}$ski carpet of dimension $n-1$ as a subset of $S^{n}$ such as described above.

For $n = 1$ this space is clearly a Cantor set. Whyburn showed (Theorem 3 of \cite{Wh}) that $M^{2}_{1}$ doesn't depend of the choice of the family $\{C_{i}\}_{i \in \N}$, and Cannon showed it for arbitrary $n \neq 4$ (Theorem 1 of \cite{Ca}). Cannon's proof also works in dimension $4$, as we discuss on the \textbf{Appendix}.

Consider the equivalence relation $\sim = \Delta M^{n}_{n-1} \cup \bigcup_{i \in \N} \partial C_{i}^{2}$. The second property of the construction of $M^{n}_{n-1}$ is equivalent to say that the quotient map $\varpi: M^{n}_{n-1} \rightarrow M^{n}_{n-1}/ \sim$ is topologically quasiconvex.

\begin{prop}\label{decompositiontheorem} $M^{n}_{n-1} / \! \sim$ is homeomorphic to $S^{n}$.
\end{prop}

\begin{proof}This is a special case of the \textbf{Decomposition Theorem}.
\end{proof}

So we can identify $M^{n}_{n-1} / \! \sim$ with the $n$-sphere.

For $A \subseteq S^{n}$, let $\sim_{A} = \Delta M^{n}_{n-1} \cup \bigcup_{\varpi(C_{i}) \notin A} \partial C_{i}^{2}$. Since $\varpi$ is a topologically quasiconvex map, it follows that $\forall A \subseteq S^{n}$,  $M^{n}_{n-1} / \! \sim_{A}$ is Hausdorff. Let $A \subseteq A'$. Consider also the quotient maps given by $\varpi_{A}: M^{n}_{n-1} \rightarrow  M^{n}_{n-1}/ \! \sim_{A}$, $\varpi_{A',A}: M^{n}_{n-1}/ \! \sim_{A'} \rightarrow  M^{n}_{n-1}/ \! \sim_{A}$ and $\varpi'_{A}: M^{n}_{n-1}/ \! \sim_{A} \rightarrow  S^{n}$. If $a \in S^{n}$ and $A \subseteq S^{n}$, we use $\varpi_{a}$, $\varpi_{A,a}$ and $\varpi'_{a}$ instead of $\varpi_{\{a\}}$, $\varpi_{A,\{a\}}$ and $\varpi'_{\{a\}}$.

There is also a quotient map $\tilde{\varpi}: S^{n} \rightarrow S^{n}$ such that $\tilde{\varpi}|_{M^{n}_{n-1}} = \varpi$ and $\forall i \in \N$, $\tilde{\varpi}(C_{i}) = \varpi(\partial C_{i})$. Analogously, if  $A \subseteq S^{n}$, then there is also a natural quotient map  $\tilde{\varpi}_{A}:  S^{n} - \bigcup_{\varpi(C_{i}) \notin A} C_{i} \rightarrow M^{n}_{n-1}/\sim_{A}$.

\begin{prop}For every finite set $A$, $M^{n}_{n-1} / \! \sim_{A}$ is homeomorphic to $S^{n}$ minus $\#A$ disjoint open tame balls.
\end{prop}

\begin{obs}By \textbf{Proposition \ref{characterizationofsphereswithoutballs}}, the space $S^{n}$ minus $\#A$ disjoint open tame balls is well defined, up to homeomorphisms.
\end{obs}

\begin{proof}Let $Z = M^{n}_{n-1}\cup \bigcup_{i \in A} C_{i}$. By the Decomposition Theorem $Z/\!\approx_{A}$ is homeomorphic to $S^{n}$, where $\approx_{A} = \sim_{A} \cup \Delta Z$. So $M^{n}_{n-1}/\!\sim_{A} = (Z - \bigcup_{i \in A} C_{i})/\! \approx_{A}$ is homeomorphic to $S^{n}$ minus $\#A$ disjoint open tame balls.
\end{proof}

Note that the last proposition and \textbf{Proposition \ref{quasiconvexity}} implies that the space $M^{n}_{n-1}$ is homeomorphic to an inverse limit of $n$-balls that quotient to the $n$-sphere collapsing their boundaries and such that the set of points of the $n$-sphere that are the image of some collapsing boundary is countable dense. It remains to prove that such inverse limit is unique, up to homeomorphisms.

\begin{defi}We say that a triangulation $\tilde{T}$ of $S^{n}$ with $n-1$-skeleton $T$ is compatible with the Sierpi\'nski carpet that we fixed, if:

\begin{enumerate}
    \item $T \subseteq  M^{n}_{n-1}$
    \item $\forall i \in \N$, $T\cap \partial C_{i} = \emptyset$ or $\partial C_{i} \subseteq T$.
    \item $T$ intersects only a finite number of spheres of $\{\partial C_{i}\}_{i \in \N}$.
    \item $\tilde{T}$ induces a triangulation of $S^{n} - \bigcup\{C_{i}: \partial C_{i} \subseteq T\}$ 
\end{enumerate} 
\end{defi}

\begin{lema}Let $\tilde{T}$ be a compatible triangulation of $S^{n}$ and $T$ its $n-1$-skeleton. So $\forall A \subseteq S^{n}$, finite, $\tilde{\varpi}_{A}(\tilde{T}) = \{\tilde{\varpi}_{A}(B): B \in \tilde{T}, B \subseteq  S^{n} - \bigcup_{i \in A} C_{i} \}$ is a CW structure of $M^{n}_{n-1} / \sim_{A}$ with its $n-1$-skeleton given by $\varpi_{A}(T)$.
\end{lema}

\begin{proof}Let $\tilde{W}$ be the CW structure of $S^{n}$ given by the triangulation $\tilde{T}$ and, for each $k \in \N$, let $W^{k}$ be its $k$-skeleton.  We have that $W^{k}$ coincides with the $k$-skeleton of $\tilde{T}$. For $i \in \N$, let $Y_{i}^{k} = \{B \in W^{k}: B \subseteq \partial C_{i}\}$. We have that $Y_{i}^{k}$ is a subcomplex of $W^{k}$. If $k \leqslant n-1$, then the quotient $\tilde{\varpi}_{A}(W^{k})$ has the CW structure given by $W^{k}$ quotiented by all subcomplexes $\{Y_{i}^{k}: \varpi(C_{i}) \notin A,\}$ (note that there is just a finite number of sets $Y_{i}^{k}$ that are not singletons). Let $B$ be a cell in $W^{n}$. If $B \cap C_{i} \neq \emptyset$ for some $i \in \N$ such that $\varpi(C_{i}) \notin A$, then $B \subseteq C_{i}$ and then we discard it. If $\forall i \in \N$ such that $\varpi(C_{i}) \subseteq A$, $B \cap C_{i} = \emptyset$, then $\tilde{\varpi}_{A}(int \ B)$ is homeomorphic to $int \ B$ by the \textbf{Decomposition Theorem}. So $\tilde{\varpi}_{A}(B)$ is a cell that we attach to $\tilde{\varpi}_{A}(W^{k})$. Doing this for every $n$-cell $B$, we get the cellular structure of the quotient.
\end{proof}

\begin{obs}Note that the induced structure on the quotient may not be a triangulation, since some $n-1$-faces, that are contained in some $\partial C_{i}$ such that $\varpi(C_{i}) \notin A$, must collapse.
\end{obs}

Analogously, we have:

\begin{lema}\label{triangulationquotient}Let $\tilde{T}$ be a compatible triangulation of $S^{n}$ and $T$ its $n-1$-skeleton. Let $A = \bigcup\{\varpi(\partial C_{i}): \partial C_{i} \subseteq T\}$. So $\tilde{\varpi}_{A}(\tilde{T}) = \{\tilde{\varpi}_{A}(B): B \in \tilde{T}, B \subseteq  S^{n} - \bigcup_{i \in A} C_{i} \}$ is a triangulation of $M^{n}_{n-1} / \sim_{A}$ with its $n-1$-skeleton given by $\varpi_{A}(T)$. \eod
\end{lema}

Let $X$ be a Hausdorff compact space and $\pi: X \rightarrow S^{n}$ a topologically quasiconvex map  satisfying:

\begin{enumerate}
    \item The set $P = \{x \in S^{n}: \#\pi^{-1}(x) > 1\}$ is countable and dense on $S^{n}$.
    \item $\forall p \in P$, the space $X_{p} = X/ \! \sim'_{p}$, where $\sim'_{p} = \Delta X \cup \bigcup_{x \neq p} \pi^{-1}(x)^{2}$, is homeomorphic to a closed ball $D^{n}$.
\end{enumerate}

By \textbf{Proposition \ref{quasiconvexity}}, these properties of $X$ are equivalent to $X$ being an inverse limit of $n$-balls that quotient to the $n$-sphere collapsing their boundaries and such that the set of points of the $n$-sphere that are images of some collapsing boundary is countable dense.

Let, for $S \subseteq S^{n}$, $\sim'_{S} = \Delta X \cup \bigcup_{x\notin S} \pi^{-1}(x)^{2}$, $X_{S} = X/ \! \sim'_{S}$ and $\pi_{S}: X \rightarrow X_{S}$ the quotient map. If $S \subseteq S' \subseteq S^{n}$, let $\pi_{S',S}: X_{S'} \rightarrow X_{S}$ and $\pi': X_{S} \rightarrow S^{n}$ be the quotient maps that commute with $\pi$. If $s \in S^{n}$ and $S \subseteq S^{n}$, we use $\pi_{s}$, $\pi_{S,s}$ and $\pi'_{s}$ instead of $\pi_{\{s\}}$, $\pi_{S,\{s\}}$ and $\pi'_{\{s\}}$.

\begin{prop}If $S$ is finite, then $X_{S}$ is homeomorphic to $S^{n}$ minus $\#S$ open tame balls.
\end{prop}

\begin{proof} Let $\{B_{s}\}_{s \in S}$ be a set of open tame balls on $S^{n}$ such that their closures are pairwise disjoint and $\forall s \in S$, $s \in B_{s}$. Let $Z = S^{n} - \bigcup_{s \in S} B_{s}$ and $Z' = \pi'^{-1}_{S}(Z)$. We have that $\pi'_{S}|_{Z'}: Z' \rightarrow Z$ is a homeomorphism.

Let $s \in S$. We have that $\pi_{S,s}(Z')$ is homeomorphic to $Z'$. Let $A_{s}$ be the preimage in $X_{S}$ of the boundary of $B_{s}$. We have that $\pi_{S,s}(A_{s})$ is tame in $X_{s}$ since $\pi'_{S}(A_{s})$ is tame in $S^{n}$ and $\pi'_{S}(A_{s}) = \pi'_{s}(\pi_{S,s}(A_{s}))$, where $\pi'_{s}$ is a homeomorphism onto its image, outside of any neighbourhood of $\pi'^{-1}_{s}(s)$. Since $\pi_{S,s}(A_{s})$ is tame in $X_{s}$, then, by the Annulus Theorem, the connected submanifold $Y$ that has $\pi_{S,s}(A_{s}) \cup \pi'^{-1}_{s}(s)$ as its boundary is homeomorphic to $S^{n-1} \times [0,1]$, which implies that $Y \cup \pi_{S,s}(Z')$ is homeomorphic to $Z'$. Let $Z'_{1} = \pi^{-1}_{S,s}(Y \cup \pi_{S,s}(Z'))$. We have that $Z'_{1}$ is homeomorphic to $Z'$. By an induction process (since $S$ is finite), we get that $X_{S}$ is homeomorphic to $Z'$ and then it is homeomorphic to $Z$, which is $S^{n}$ minus $\#S$ open tame balls.
\end{proof}

\begin{prop}\label{cannonissue}(Cannon - Lemma 1 of \cite{Ca}) Let $\tilde{T}$ be a compatible triangulation of $S^{n}$, $T$ be its $n-1$-skeleton and  $\epsilon > 0$.  Let $i \in \N$. Then there exists a subdivision $\tilde{T}'$ of $\tilde{T}$ such that it is compatible, $T' \cap \bigcup_{j \in \N} C_{j} = (T \cap \bigcup_{j \in \N} C_{j}) \cup \partial C_{i}$, where $T'$ is the  $n-1$-skeleton of $\tilde{T}'$, and for every $n$-simplex $D$ in $\tilde{T}'$, $diam \ D < \epsilon$.
\end{prop}

\begin{obs}For $n = 2$, the proposition above is due to Whyburn (Lemma 1 of \cite{Wh}).

This proposition is stated and proved in \cite{Ca} with the restriction that $n \neq 4$, since some theorems used in Cannon's proof had this restriction at that time. However, Cannon's proof works without this restriction, as we comment on the \textbf{Appendix}.
\end{obs}

\begin{teo}\label{blowupissierpinski}$X$ is homeomorphic to $M^{n}_{n-1}$.
\end{teo}

\begin{obs}This theorem is proved for $n = 2$ in \cite{TW2}. Our methods used in this proof are similar to theirs.
\end{obs}

\begin{proof}Here we do a back-and-forth argument.

Let $\tilde{T}_{0}$ be a compatible triangulation of $S^{n}$ and its $n-1$-skeleton $T_{0}$.

Since $S^{n}$ is countable dense homogeneous, we can assume that $M^{n}_{n-1}/ \! \sim = S^{n}$ and $P = \{x \in S^{n}: \#\varpi^{-1}(x) > 1\}$ (remember that $P$ was originally defined for $X$ and not for $M^{n}_{n-1}$).

Let $S_{0} = \varpi(\bigcup_{j \in \N}C_{j}\cap T_{0})$ and $R_{0} = S_{0}$. By \textbf{Proposition \ref{characterizationofsphereswithoutballs}}, there is a homeomorphism $f_{0}:M^{n}_{n-1}/ \!\! \sim_{S_{0}} \rightarrow X_{R_{0}}$. We have that $\forall x \in \varpi_{S_{0}}(T_{0})$, $\#\pi_{R_{0}}^{-1}(f_{0}(x)) = 1$. Then $\pi_{R_{0}}|_{\pi_{R_{0}}^{-1}(\varpi_{S_{0}}(T_{0}))}$ is a homeomorphism onto its image, which implies that there is an embedding $\tilde{f}_{0}: T_{0} \rightarrow X$ that commutes the diagram:

$$ \xymatrix{ T_{0} \ar@{(->}[r]^<{\ \ \ \ \ \ \  \varpi|_{T_{0}}} \ar@{-->}[d]_{\tilde{f}_{0}} & M^{n}_{n-1}/ \!\! \sim_{S_{0}} \ar[d]_{f_{0}} \\
            X  \ar[r]_{\pi} & X_{R_{0}} }$$

Suppose that we have a family of triangulations $\{\tilde{T}_{k}: 0 \leqslant k \leqslant 2i\}$ that are compatible and, together with their respective $n-1$-skeletons $T_{k}$, satisfies:

\begin{enumerate}
    \item $\forall k \leqslant 2i$, $\tilde{T}_{k}$ a subdivision of $\tilde{T}_{k-1}$.
    \item $\bigcup_{j = 1}^{i} \partial C_{j} \subseteq T_{2i}$.
\end{enumerate}

Let $S_{k} = \varpi(\bigcup_{j \in \N}C_{j}\cap T_{k})$. We have that $S_{k}\supseteq S_{k-1}$.

Suppose also that we have a family of embeddings $\tilde{f}_{k}: T_{k} \rightarrow X$ and a family of sets $R_{k} = \pi(\tilde{f}_{k}(\bigcup_{j \in \N}C_{j}\cap T_{k}))$, with $0 \leqslant k \leqslant 2i$, satisfying:

\begin{enumerate}
    \item[3.] $\forall k \leqslant 2i$, $\bigcup\{\pi(\partial C_{j}): 2j \leqslant k\} \subseteq R_{k}$.
\end{enumerate}

We have that $\forall k \leqslant 2i$, $R_{k} \supseteq R_{k-1}$.

Finally, suppose that we have families of homeomorphisms $\\ f_{k}: M^{n}_{n-1}/\!\sim_{S_{k}} \rightarrow X_{R_{k}}$, $\hat{f}_{k}: S^{n} \rightarrow S^{n}$ and, $\forall s \in S_{k}$, $f_{s,k}: M^{n}_{n-1}/ \! \sim_{s} \rightarrow X_{\hat{f}_{k}(s)}$, with  $0 \leqslant k \leqslant 2i$, satisfying:

\begin{enumerate}
    \item[4.] The diagram commutes, for every $s \in S_{k}$: 
    
    $$ \xymatrix{  T_{k} \ar@{(->}[rr]^-{ \varpi_{S_{k}}|_{T_{k}}} \ar@{-->}[d]_{\tilde{f}_{k}} & & M^{n}_{n-1}/ \!\! \sim_{S_{k}}  \ar[d]_{f_{k}} \ar[r]^{\varpi'_{S_{k},s}} &  M^{n}_{n-1}/ \! \sim_{s} \ar[d]_{f_{s,k}} \ar[r]^-{\varpi'_{s}} & S^{n} \ar[d]_{\hat{f}_{k}} \\
           X  \ar[rr]_{\pi_{R_{k}}} & & X_{R_{k}} \ar[r]_{\pi'_{R_{k},\hat{f}_{k}(s)}} & X_{\hat{f}_{k}(s)} \ar[r]_{\pi'_{\hat{f}_{k}(s)}} & S^{n} }$$
           
    \item[5.] $\forall k \leqslant 2i$, the diameter of every simplex of $\tilde{\varpi}_{s}(\tilde{T}_{k})$,  $f_{s,k-1}\circ\tilde{\varpi}_{s}(\tilde{T}_{k})$, for every choice of $s \in S_{k-1}$, and $\tilde{\varpi}(\tilde{T}_{k})$ and $\hat{f}_{k-1}\circ \tilde{\varpi}(\tilde{T}_{k})$ is less than $\frac{1}{2^{k}}$.
\end{enumerate}

Observe that the maps $f_{s,k}$ and $\hat{f}_{k}$ are uniquely defined by $f_{k}$ and the commutative diagram.

Let $\tilde{T}_{2i+1}$ be a subdivision of $\tilde{T}_{2i}$ that is compatible and let $T_{2i+1}$ its $n-1$-skeleton, satisfying:

\begin{enumerate}
    \item[2.] $\partial C_{i+1} \subseteq T_{2i+1}$.
    \item[5.] The diameter of every simplex of $\tilde{\varpi}_{s}(\tilde{T}_{2i+1})$,  $f_{s,2i}\circ\tilde{\varpi}_{s}(\tilde{T}_{2i+1})$, for every choice of $s \in S_{2i}$, and $\tilde{\varpi}(\tilde{T}_{2i+1})$ and $\hat{f}_{2i}\circ \tilde{\varpi}(\tilde{T}_{2i+1})$ is less than $\frac{1}{2^{2i+1}}$.
\end{enumerate}

We can do that by subdividing $\tilde{T}_{2i}$, using the last proposition, to have simplexes small enough such that the uniform continuity property of these maps gives induced triangulations satisfying (5).

If $\varpi(\partial C_{i+1}) \in S_{2i}$, then we take $S_{2i+1} = S_{2i}$, $R_{2i+1} = R_{2i}$, $\tilde{T}_{2i+1} = \tilde{T}_{2i}$, $f_{2i+1} = f_{2i}$ and $\hat{f}_{2i+1} = \hat{f}_{2i}$. In this case, all six conditions are satisfied.

Suppose that $\varpi(\partial C_{i+1}) \notin S_{2i}$. Let  $S_{2i+1} = \varpi(\bigcup_{j \in \N}C_{j}\cap T_{2i+1})$. We have that $S_{2i+1}\supseteq S_{2i}$ and $\varpi(\partial C_{2i+1}) \in S_{2i+1}$. For every $s \in S_{2i+1}-S_{2i}$, take $E_{s}$ as the $n$-simplex in $\tilde{T}_{2i}$ such that $s \in \tilde{\varpi}(E_{s})$, choose $p_{s} \in \tilde{\varpi}(E_{s}) \cap P$ (since  $s \in \tilde{\varpi}(E_{s})$, then $int(\tilde{\varpi}(E_{s}))\neq \emptyset$) and define $R_{2i+1} = R_{2i} \cup \{p_{s}: s \in S_{2i+1}-S_{2i}\}$. We have that $R_{2i+1} \supseteq R_{2i}$ and the condition (3) is automatically satisfied (it is necessary to add something to the set to satisfy this condition only on even steps).

By \textbf{Lemma \ref{triangulationquotient}}, $\tilde{T}_{2i}$ induces a triangulation in $M^{n}_{n-1}/ \! \! \sim_{S_{2i+1}}$.  Let $D$ be a $n$-simplex in $\tilde{T}_{2i}$ and $S_{D} = P \cap \tilde{\varpi}(int \ D)$. We need a homeomorphism $f_{D}: (D\cap M^{n}_{n-1})/ \!\! \sim_{S_{2i+1}} \rightarrow Y_{D}$, where $Y_{D}$ is the closure of the connected component of $X_{R_{2i+1}} - \pi_{R_{2i+1}}(\tilde{f}_{2i}(T_{2i}))$ that contains $\pi'^{-1}_{R_{2i+1}}(\hat{f}_{2i}(S_{D}))$, such that $f_{D}|_{\varpi_{S_{2i+1}}(\partial D)} = \pi_{R_{2i+1}}\circ \tilde{f}_{2i}\circ \varpi^{-1}_{S_{2i+1}}|_{\varpi_{S_{2i+1}}(\partial D)}$. Every term of this composition is well defined and continuous, when restricted to suitable domains. So $f_{D}|_{\varpi_{S_{2i+1}}(\partial D)}$ is well defined, continuous and, by the choice of the points $p_{s}$, the spaces $(D\cap M^{n}_{n-1})/ \!\! \sim_{S_{2i+1}}$ and $Y_{D}$ are homeomorphic. Then, by \textbf{Proposition \ref{characterizationofsphereswithoutballs}}, there exists a homeomorphism $f_{D}$ that extends $\pi_{R_{2i+1}}\circ \tilde{f}_{2i}\circ \varpi^{-1}_{S_{2i+1}}|_{\varpi_{S_{2i+1}}(\partial D)}$ and such that $\forall s \in S_{2i+1} \cap S_{D}$, $f_{D}(\varpi'^{-1}_{S_{2i+1}}(s)) = \pi'^{-1}_{R_{2i+1}}(p_{s})$. By \textbf{Proposition \ref{cotabledensehomogeneous}} we can suppose also that $f_{D}(S_{2i+1}(S_{D})) = \pi'^{-1}_{R_{2i+1}}(S_{D})$. Take $f_{2i+1}: M^{n}_{n-1}/\! \sim_{S_{2i+1}} \rightarrow X_{R_{2i+1}}$ as $f_{2i+1}(x) = f_{D}(x)$ if $x \in (D\cap M^{n}_{n-1})/\! \sim_{S_{2i+1}}$. Since $f_{2i+1}|_{T_{2i}}$ does not depend of the choice of the $n$-simplex $D$ and two $n$-simplexes intersect only in $T'_{2i}$, we have that $f_{2i+1}$ is well defined and a homeomorphism. It is immediate that there is a homeomorphism $\hat{f}_{2i+1}: S^{n} \rightarrow S^{n}$ that commutes the diagram:

$$ \xymatrix{ M^{n}_{n-1}/ \! \! \sim_{S_{2i+1}} \ar[d]_{f_{2i+1}} \ar[r]^-{\varpi'_{S_{2i+1}}} & S^{n} \ar[d]_{\hat{f}_{2i+1}} \\
            X_{R_{2i+1}} \ar[r]_{\pi'_{R_{2i+1}}} & S^{n} }$$

Observe that, since it holds for every $n$-simplex $D$, we have that  $\hat{f}_{i}(P) = P$. We have also that $\hat{f}_{2i+1}|_{\varpi(T_{2i})} = \hat{f}_{2i}|_{\varpi(T_{2i})}$.

Define $\tilde{f}_{2i+1}: T_{2i+1} \rightarrow X$ as $\tilde{f}_{2i+1} = \pi^{-1}_{R_{2i+1}}\circ f_{2i+1} \circ \varpi_{S_{2i+1}}|_{T_{2i+1}}$. It is clear that $\tilde{f}_{2i+1}$ is an embedding that extends $\tilde{f}_{2i}$.

For every choice of $s \in S_{2i+1}$, let $f_{s,2i+1}: M^{n}_{n-1}/ \! \sim_{s} \rightarrow X_{\hat{f}_{2i+1}(s)}$ be the homeomorphism that commutes the diagram (i.e., satisfies (4)):

$$ \xymatrix{  T_{2i+1} \ar@{(->}[rr]^-{ \varpi_{S_{2i+1}}|_{T_{2i+1}}} \ar@{-->}[d]_{\tilde{f}_{2i+1}} & &  M^{n}_{n-1}/ \!\! \sim_{S_{2i+1}}  \ar[d]_{f_{2i+1}} \ar[r]^{\varpi'_{S_{2i+1},s}} &  M^{n}_{n-1}/ \!\! \sim_{s} \ar[d]_{f_{s,2i+1}} \ar[r]^-{\varpi'_{s}} & S^{n} \ar[d]_{\hat{f}_{2i+1}} \\
            X  \ar[rr]_{\pi_{R_{2i+1}}} & &  X_{R_{2i+1}} \ar[r]_{\pi'_{R_{2i+1},s}} & X_{\hat{f}_{2i+1}(s)} \ar[r]_{ \ \ \ \ \ \ \ \ \ \pi'_{\hat{f}_{2i+1}(s)}} & S^{n} }$$

Now, if we have $\tilde{T}_{2i+1}$, $S_{2i+1}$, $R_{2i+1}$, $f_{2i+1}$, $\hat{f}_{2i+1}$ and $\forall s \in S_{2i+1}$, $f_{s,2i+1}$, then we do an entirely analogous construction, but requiring that $\pi(\partial C_{i}) \in R_{2i+2}$.

By \textbf{Proposition \ref{aproximaçãoportriangulacoes}}, there are homeomorphisms $\hat{f}: S^{n} \rightarrow S^{n}$ and, $\forall s \in P$, $f_{s}: M^{n}_{n-1}/ \!\! \sim_{s} \rightarrow X_{\hat{f}(s)}$ such that $\forall i \in \N$, $\hat{f}|_{\varpi(T_{i})} = \hat{f}_{i}|_{\varpi(T_{i})}$ and $f_{s}|_{\varpi_{s}(T_{i})} = f_{s,i}|_{\varpi_{s}(T_{i})}$. Let $T = \bigcup_{i \in \N}T_{i}$. Since it commutes for every $i \in \N$, the following diagram commutes (for every $s \in P$):

$$ \xymatrix{  \varpi_{s}(T_{i}) \ar[d]_{f_{s}} \ar[r]^{\varpi'_{s}} & S^{n} \ar[d]_{\hat{f}} \\
            f_{s}(\varpi_{s}(T_{i})) \ar[r]_-{\pi'_{\hat{f}(s)}} & S^{n} }$$
            
Since $T$ contains $\bigcup\{C_{i}: i \in \N\}$, it is dense on $M^{n}_{n-1}$ and $\varpi_{s}(T)$ is dense on $M^{n}_{n-1}/ \! \sim_{s}$, which implies that the following diagram commutes:  

$$ \xymatrix{  M^{n}_{n-1}/ \!\! \sim_{s} \ar[d]_{f_{s}} \ar[r]^-{\varpi'_{s}} & S^{n} \ar[d]_{\hat{f}} \\
            X_{\hat{f}(s)} \ar[r]_{\pi'_{\hat{f}(s)}} & S^{n} }$$
            
By the back-and-forth construction, $\hat{f}(P) = P$. Then the family of homeomorphisms $\{\hat{f}, f_{i}, i \in \N\}$ induces an homeomorphism $\tilde{f}: M^{n}_{n-1} \rightarrow X$.
\end{proof}

\section*{Appendix}

Here we show a brief comment on what works in the Approximation Theorem for Cellular Maps on dimension 4 that should be enough for Cannon's lemma that we used before (\textbf{Proposition \ref{cannonissue}}).

\begin{defi}Let $X$ be a topological space, $(Y,d)$ a metric space and $\pi: X \rightarrow Y$ a continuous map. We say that $\pi$ is a near homeomorphism if for every $\epsilon > 0$, there exists a homeomorphism $f: X \rightarrow Y$ such that for every $x \in X$, $d(\pi(x),f(x)) < \epsilon$.   
\end{defi}

\begin{prop}\label{nearfreedman} (Freedman, Theorem 9.1' of \cite{Fre}) Let $\pi: S^{n} \rightarrow S^{n}$ be a topologically quasiconvex map such that the set $\{x \in S^{n}: \# \pi^{-1}(x) > 1\}$ is nowhere dense in $S^{n}$. Then $\pi$ is a near homeomorphism (for any metric on $S^{n}$).
\end{prop}

\begin{prop}\label{aproximantionfixingcompact} (Corollary 7.1 of \cite{Fre}) Let $\pi: M \rightarrow N$ be a near homeomorphism between compact manifolds. Let $C$ be a compact subset of $M$ such that for all $p \in C$, $\{p\}$ is saturated. Then, for every $\epsilon > 0$, there exists a homeomorphism $f: M \rightarrow N$ such that for every $x \in M$, $d(\pi(x),f(y)) < \epsilon$ and $f|_{C} = \pi|_{C}$.
\end{prop}

Combining the two propositions, we get:

\begin{prop}\label{cannonfixed} Let $\pi: S^{n} \rightarrow S^{n}$ be a topologically quasiconvex map such that the set $\{x \in S^{n}: \# \pi^{-1}(x) > 1\}$ is nowhere dense in $S^{n}$, $U = \{U_{\alpha}\}_{\alpha \in \Gamma}$ be an open cover of $S^{n}$ where all sets are saturated and $C$ be a compact subset of $S^{n}$ such that for all $p \in C$, $\{p\}$ is saturated. Then there exists a homeomorphism $g: S^{n} \rightarrow S^{n}$ such that $g \circ \pi|_{C} = id_{C}$ and for every $x \in S^{n}$, there exists $\alpha \in \Gamma$ such that $x, g \circ \pi(x) \in U_{\alpha}$.
\end{prop}

\begin{proof}Fix a metric $d$ on $S^{n}$. The set $\pi(U) = \{\pi(U_{\alpha})\}_{\alpha \in \Gamma}$ is an open cover of $N$. Let $\epsilon > 0$ be the Lebesgue number of the covering $\pi(U)$. By the \textbf{Proposition \ref{nearfreedman}}, the map $\pi$ is a near homeomorphism, which implies, by \textbf{Proposition \ref{aproximantionfixingcompact}}, that there exists a homeomorphism 
$f: M \rightarrow N$ such that $f|_{C} = \pi|_{C}$ and $\forall x \in S^{n}$, $d(\pi(x),f(x)) < \epsilon$. 

Let $x \in M$. Take $y = f^{-1} \circ \pi(x)$. We have that $d(f(y), \pi(y)) < \epsilon$. But $f(y) = \pi(x)$, which implies that $d(\pi(x), \pi(y)) < \epsilon$. Since $\epsilon$ is the Lebesgue number of the cover $\pi(U)$, there exists $\alpha \in \Gamma$ such that $\pi(x), \pi(y) \in \pi(U_{\alpha})$. Since $U_{\alpha}$ is saturated, we have that $x,y \in U_{\alpha}$. Then $x, f^{-1}\circ \pi(x) \in U_{\alpha}$. We have also that $f|_{C} = \pi|_{C}$, which implies that $f^{-1} \circ \pi|_{C} = id_{C}$.  Thus, $f^{-1}$ is the homeomorphism that we want.   
\end{proof}

\begin{obs}Note that if we take $C$ as a $n-1$-sphere embedded in $S^{n}$, then the fact that there exists an homeomorphism $f: S^{n} \rightarrow S^{n}$ that sends $C$ to $\pi(C)$ implies that $C$ is tame if and only if $\pi(C)$ is tame.
\end{obs}

\textbf{Proposition \ref{cannonfixed}} and the \textbf{Remark} are enough to replace the Approximation Theorem for Cellular Maps and its corollary on  Cannon's proof of \textbf{Proposition \ref{cannonissue}}. Since the Annulus Theorem also works in dimension $4$ (Quinn, \cite{Qu}), then there are no more restrictions on the dimension. Thus, \textbf{Proposition \ref{cannonissue}} works in dimension $4$ as well. For the same reason, Cannon's theorem that says that the space $M^{n}_{n-1}$ does not depend on the choices of removed open balls (Theorem 1 of \cite{Ca}) also works in dimension $4$.

\end{document}